\documentclass[12pt]{amsart}

\DeclareMathOperator{\Gal}{Gal}%

\DeclareMathOperator{\res}{res}
\DeclareMathOperator{\cor}{cor}

\numberwithin{equation}{section}
\newtheorem{theorem}[equation]{Theorem}
\newtheorem{corollary}[equation]{Corollary}
\newtheorem{proposition}[equation]{Proposition}
\newtheorem{lemma}[equation]{Lemma}
\newtheorem{remark}[equation]{Remark}

\newtheorem*{definition*}{Definition}
\newtheorem*{remark*}{Remark}

\newcommand{\F}{\mathbb{F}}

\newcommand{\N}{\mathbb{N}}

\newcommand{\Z}{\mathbb{Z}}

\newcommand{\im}{\text{\rm{im}}}
\newcommand{\comment}[1]{}

\usepackage{amssymb}
\usepackage{xypic}

\begin{document}

\title[Galois Cohomology for Embeddable Cyclic Extensions]{Galois
Module Structure of Galois Cohomology for Embeddable Cyclic Extensions of
Degree $p^n$}

\author[N.~Lemire]{Nicole Lemire}
\address{Department of Mathematics, Middlesex College, \ University
of Western Ontario, London, Ontario \ N6A 5B7 \ CANADA}
\email{nlemire@uwo.ca}

\author[J.~Min\'{a}\v{c}]{J\'an Min\'a\v{c}}
\address{Department of Mathematics, Middlesex College, \ University
of Western Ontario, London, Ontario \ N6A 5B7 \ CANADA}
\email{minac@uwo.ca}

\author[A.~Schultz]{Andrew Schultz}
\address{Department of Mathematics, University of Illinois at Urbana-Champaign, 1409 W. Green Street, Urbana, IL \ 61801 \ USA}
\email{acs@math.uiuc.edu}

\author[J.~Swallow]{John Swallow}
\address{Department of Mathematics, Davidson College, Box 7046,
Davidson, North Carolina \ 28035-7046 \ USA}
\email{joswallow@davidson.edu}

\begin{abstract}
Let $p>2$ be prime, and let $n,m \in \N$ be given.  For cyclic extensions $E/F$ of degree $p^n$ that contain a primitive $p$th root of unity, we show that the associated $\F_p[\Gal(E/F)]$-modules $H^m(G_E,\mu_p)$ have a sparse decomposition.  When $E/F$ is additionally a subextension of a cyclic, degree $p^{n+1}$ extension $E'/F$, we give a more refined $\F_p[\Gal(E/F)]$-decomposition of $H^m(G_E,\mu_p)$.  
\end{abstract}

\date{\today}




\maketitle

\parskip=12pt plus 2pt minus 2pt

\section{Introduction}

Absolute Galois groups capture a great deal of the arithmetic and algebraic properties of their underlying fields, though they are notoriously intractable to compute.  For a given field $E$, one must often be satisfied with studying invariants attached to the corresponding absolute Galois group $G_E$, and in this respect the Galois cohomology groups $H^i(G_E,A)$ for various $G_E$-modules $A$ are frequent subjects of investigation.  Of particular interest are the groups $H^m(G_E,\mu_p)$ for a fixed prime $p$, where $\mu_p$ represents the group of $p$th roots of unity in $G_E$.

When $E$ is itself a Galois extension of a field $F$, the action of $\Gal(E/F)$ on $E^\times$ induces a natural action on $H^m(G_E,\mu_p)$.  Combined with the $\F_p$-action on these cohomology groups, this naturally leads one to study these Galois cohomology groups as $\F_p[\Gal(E/F)]$-modules.  In particular, one expects that this Galois module structure will provide insight into the corresponding absolute Galois group $G_E$.

This program has been carried out in several cases where $\Gal(E/F) \simeq \Z/p^n\Z$ and $E$ contains a primitive $p$th root of unity $\xi_p$.  In particular, the case $n=m=1$ was resolved in \cite{MS2}, $m=1$ and $n\geq 1$ (without the restriction $\xi_p \in E$) in \cite{MSS1}, and $m\geq 1$ and $n=1$ in \cite{LMS3}.  As desired, these computed module structures have already led to some interesting results on the structure of absolute Galois groups: automatic realization results in \cite{MS3,MSS2}, a generalization of Schreier's Theorem in \cite{LLMS3}, a connection with Demu\v{s}kin groups in \cite{LLMS2}, an interpretation of cohomological dimension in \cite{LMS}, and a characterization of certain groups which cannot appear as absolute Galois groups in \cite{BLMS}.

The goal of this paper is to begin the investigation of a unified understanding of the structures already computed by determining some important results in the case $m\geq 1$ and $n \geq 1$.  We shall focus on the case $p>2$ in this paper.  In much the same way that this problem is the unification of the problems considered in \cite{LMS3} and \cite{MSS1}, so too will the methodology in our solution be a combination of their individual strategies.  Indeed, careful refinements of the arguments from \cite{LMS3}, together with the appropriate module-theoretic results, will already be enough to give us the following

\begin{theorem}\label{th:coarse.decomposition}
Let $p>2$ be a given prime.  If $\Gal(E/F) \simeq \Z/p^n\Z$ and $\xi_p \in E$, then the $\F_p[\Gal(E/F)]$-module $H^m(G_E,\mu_p)$ is a direct sum of indecomposable summands which are either of dimension $p^n$ or of dimension at most $2p^{n-1}$.
\end{theorem}

Since there are $p^n$ isomorphism classes of indecomposable $\F_p[G]$-modules --- one for each cyclic submodule of $\F_p$-dimension $i$, $1 \leq i \leq p^n$ --- this result shows that the decomposition of $H^m(G_E,\mu_p)$ is relatively sparse.

A more refined decomposition is available, however, if we impose an additional assumption on the extension $E/F$.  When $\Gal(E/F) \simeq \Z/p^n\Z$ and $\xi_p \in E$, we say that $E/F$ is an \emph{embeddable} extension if $E/F$ is an intermediate extension in a larger Galois extension $E'/F$ so that $$\xymatrix{\Gal(E'/F) \ar[r]\ar[d]^{\simeq}& \Gal(E/F)\ar[d]^{\simeq}\\\Z/p^{n+1}\Z \ar[r] & \Z/p^n\Z,}$$ where the horizontal arrows are the natural projections.

In the case of embeddable extensions, we can then use results from \cite{MSS1} --- particularly the properties of so-called ``exceptional'' elements of $E$ (see Proposition \ref{prop:exceptional.element.criteria}) --- to give the following result.  In the statement of the result, we use $E_j$ to denote the intermediate field of degree $p^j$ over $F$ within the extension $E/F$.

\begin{theorem}\label{th:embeddable.decomposition}
Let $p>2$ be a given prime.  If $E/F$ is an embeddable extension and $a_n$ is an exceptional element, then as an $\F_p[\Gal(E/F)]$-module we have
$$H^m(G_E,\mu_p) \simeq X_0 \oplus X_1 \oplus \cdots \oplus X_{n-1} \oplus Y_0 \oplus Y_1 \oplus \cdots \oplus \oplus Y_{n},$$ where
\begin{itemize}
\item for each $0 \leq i \leq n$, both $Y_i$ and $X_i$ are direct sums of indecomposable modules of dimension $p^i$, with $Y_i \subseteq  \res_{E/E_i}\left(H^m(G_{E_i},\mu_p)\right)$ and $X_i \subseteq  (a_n) \cup \res_{E/E_i}\left(H^{m-1}(G_{E_i},\mu_p)\right)$; and
\item for each $i \geq 0$, $\res_{E/F}\left(\cor_{E_i/F}\left(H^m(G_{E_i},\mu_p)\right)\right) = (Y_i \oplus \cdots \oplus Y_n)^G$.
\end{itemize}
\end{theorem}

Though the strategies for embeddable extensions cannot be translated directly into a decomposition of the Galois module structure of $H^m(G_E,\mu_p)$ when $E/F$ is not embeddable, this is nonetheless an important step towards resolving the more general case.  As an indication of this, we note that for a non-embeddable extension $E/F$, any proper subextension is embeddable.  For ``bottom-up'' inductive arguments (i.e., those which rely on studying subextensions which share a common base field), then, the embeddable case is of critical importance.  These kinds of arguments were already used to great effect in resolving the case $m=1,n>1$ in \cite{MSS1}, so it is likely that a resolution of the general (non-embeddable) case for higher cohomology will also include this strategy.



Section \ref{sec:preliminaries} outlines the basic ingredients necessary for the proofs of the main theorems, recalling important facts about Galois cohomology, module theory and field theory.  Section \ref{sec:Gamma} then gives a description of a submodule $\Gamma(m,n) \subseteq  H^{m-1}(G_{E_{n-1}},\mu_p)$ which is critical for our inductive approach.  Building on these results, Section \ref{sec:fixed.elements.are.norms} describes the major technical results needed to provide a proof of Theorem \ref{th:embeddable.decomposition} in Section \ref{sec:proof.of.theorem}.

\begin{remark}
Though the proof of Theorem \ref{th:embeddable.decomposition} relies on working in an embeddable extension, the other machinery we develop holds for extensions $E/F$ with $\Gal(E/F) \simeq \Z/p^n\Z$ and $\xi_p \in E$ ($p>2$ a prime) without insisting on embeddability.
\end{remark}

The case $p=2$ requires special treatment and is a work in progress.

\section{Preliminary Results}\label{sec:preliminaries}


\subsection{Reduced Milnor $K$-theory}
Though we've phrased our results in the language of Galois cohomology, the driving force in these proofs is the connection between these cohomology groups and reduced Milnor $K$-theory that was first described by the so-called Bloch-Kato conjecture, the content of which is stated in the following (recently proven) theorem. We shall denote by $k_mE$ the reduced Milnor $K$-groups $K_mE/pK_mE$.



\begin{theorem}\label{th:BlochKato}
For a field $E$ containing a primitive $p$th root of unity $\xi_p$, the natural map $$k_mE \to H^m(G_E,\mu_p)$$ is an isomorphism.  Moreover, if $G = \Gal(E/F)$ for some subextension $F$, then the isomorphism is $G$-equivariant on the two $G$-modules.
\end{theorem}

The process of proving the Bloch-Kato conjecture began with Merkurjev and Suslin \cite{MS1}, who verified the case $m=2$ for all primes $p$.  The case where $p=2$ and $m$ is arbitrary was resolved by Voevodsky \cite{Vo1}. Recently Rost and Voevodsky together with Weibel's patch proved the Bloch-Kato conjecture for all $p$ and $m$.  For details, see \cite{Ro1,Ro2,Vo2,HW,W1,W2,W3}. In what follows, we will employ Theorem \ref{th:BlochKato} without mention to identify Galois cohomology and reduced Milnor K-theory.

The strategy we employ will require generalizations to $k$-theory of some well-known results from field theory, namely Hilbert's Theorem 90 and Kummer theory.  In this new setting, both of these results deal with extensions $E/F$ that are degree $p$ and Galois.   In the results that follow, $N_{E/F}$ denotes the map induced on $K$-theory by the field norm, and $\iota_{E/F}$ denotes the map induced on $K$-theory by inclusion.  

The results below can be deduced from the papers cited above in the proof of the Bloch-Kato conjecture, and they are in fact important parts of the proof.  An exposition of the precise results leading to these Propositions is contained in Section 2 of \cite{LMSS}.

\begin{proposition}[Hilbert 90 for $K$-theory]\label{prop:hilbert.90.for.k.theory}
If $\Gal(E/F) = \langle \sigma \rangle \simeq \Z/p\Z$, then the sequence
\begin{equation}\label{eq:hilbert.90.for.k.theory}
K_mE \xrightarrow{\ \sigma-1\ } K_mE \xrightarrow{N_{E/F}} K_mF
\end{equation}
is exact.
\end{proposition}


\begin{proposition}[Kummer Theory]
Continuing with the assumptions of Proposition \ref{prop:hilbert.90.for.k.theory}, suppose that $\xi_p \in E$ and that $E = F(\root{p}\of{a})$ for $a \in F^\times$. Then the sequence
\begin{equation}\label{eq:kummer.for.k.theory}
k_{m-1}E \xrightarrow{N_{E/F}} k_{m-1}F \xrightarrow{\{a\}\cdot -} k_mF \xrightarrow{\iota_{E/F}} k_mE
\end{equation}
is exact.
\end{proposition}

Finally, we need a result which allows one to easily compute the norm of a special class of symbols.  

\begin{proposition}[Projection Formula, {\cite[Chap.~9, Thm.~3.8]{FV}}]
Let $E/F$ be a Galois extension of fields, and let $e \in E^\times$ and $\gamma \in K_{m-1}F$ be given.  Then
\begin{equation}\label{eq:projection.formula}
N_{E/F}\left(\{e\}\cdot\iota_{E/F}(\gamma)\right) = \{N_{E/F}(e)\}\cdot \gamma.
\end{equation}
\end{proposition}


\subsection{Field Results} 
Since our focus is on embeddable extensions, there are several simple Galois-theoretic consequences which will be useful to record.

Recall that we denote by $E_i$ the intermediate field of degree $p^i$ over $F$ within $E/F$.  Hence we will interchangeably refer to $E$ as $E_n$ and $F$ as $E_0$.  We will write $G_i$ for the quotient group $\Gal(E_i/F) \simeq \Z/p^i\Z$ and $H_i$ for the subgroup $\Gal(E/E_i) \simeq \Z/p^{n-i}\Z$.  For convenience, we will carry this notation over to abbreviate relevant inclusion and norm maps: $\iota_j^i$ will denote $\iota_{E_i/E_j}$ and $N^i_j$ will denote $N_{E_i/E_j}$, both for fields and their $K$-theory.

Since we have assumed $\xi_p \in E$ we must have $\xi_p \in E_i$ for every $0 \leq i \leq n$, and so it follows by Kummer Theory that for every $0 \leq i \leq n-1$ we may find elements $a_i \in E_i$ so that $E_{i+1} = E_i(\root{p}\of{a_i})$ .  In fact, it was shown in \cite[Prop.~1]{MSS1} that these $a_i$ can be selected to satisfy the following norm compatibility criteria: \begin{equation}\label{eq:normcompatibility}N_j^i a_i = a_j \quad \mbox{ for any } j \leq i \leq n-1.\end{equation} It is also shown that for $0\leq i \leq n-1$, the $p$th power class of each of these $a_i$ is fixed by its respective Galois group:
\begin{equation}\label{eq:ai.are.fixed}
\tau(a_i) \in a_iE_i^{\times p} \quad \mbox{ for every } \tau \in G_i.
\end{equation}

In \cite{MSS1}, exceptional elements for the extension $E/F$ are defined as a kind of ``minimal'' extension of the above equations to $i=n$.  The definition there is expressed in terms of classes of elements in $E^\times/E^{\times p}$; we now present an equivalent formulation for elements in $E^\times$.  For a general field $E/F$ with $\Gal(E/F) \simeq \Z/p^n\Z$ and containing $\xi_p$, an exceptional element $a_n \in E^\times$ for the extension $E/F$ is an element with $N_{E/F}(a_n) \in E^{\times p} \setminus F^{\times p}$, and such that $a_n^{\sigma-1} \in E_{i(E/F)}^\times E^{\times p}$, where $$i(E/F) = \min_{N_{E/F}(a) \in E^{\times p} \setminus F^{\times p}} \left\{i: a^{\sigma-1} \in E_i^\times E^{\times p}\right\}.$$ (Here, $E_{-\infty}$ is taken to be $\{1\}$.)  Hence the possible values for $i(E/F)$ are from the set $\{-\infty, 0, 1, \cdots, n\}$;  in \cite[Theorem 3]{MSS1}, we show that in fact $i(E/F) \leq n-1$.

One can show that exceptional elements exist under the hypothesis $p>2$ (\cite[Prop.~2]{MSS1}), and that the cyclic $\F_p[\Gal(E/F)]$-submodule generated by an exceptional element has $\F_p$-dimension $p^{i(E/F)}+1$ (\cite[Prop.~7]{MSS1}).  For an exceptional element $a_n$, this fact, together with the condition $N_{E/F}(a_n) \in E^{\times p} \setminus F^{\times p}$,  ensures that $N^n_j(a_n^t)= a_j e_j^p$ for some $t \in \Z\setminus p\Z$ and $e_j \in E_j^{\times}$(\cite[Lemma 8]{MSS1}).  (One might naturally think to select $a_n$ so that $N^n_j(a_n) = a_j$ in analogy to (\ref{eq:normcompatibility}), but our weaker condition is chosen to account for the set $\{a_i\}$ being non-canonical.) 

Our embeddability hypothesis is equivalent to the condition that $i(E/F) = -\infty$, so that $a_n^{\sigma-1} \in E^{\times p}$ in this case.  This condition will become important in the final stages of our proof of Theorem \ref{th:embeddable.decomposition}, but isn't necessary for other results we describe. 

For our purposes, the relevant properties of exceptional elements are outlined in the following


\begin{proposition}\label{prop:exceptional.element.criteria}
For each $1 \leq i < n$, the element $a_i$ is exceptional for the extension $E_i/F$.
Any exceptional element $a_n$ of an embeddable extension $E/F$ satisfies $N^n_{n-1}(a_n) = a_{n-1}^te^p$ for some $t \in \Z\setminus p\Z$ and $e \in E_{n-1}^{\times}$, and furthermore $a_n^{\sigma-1} \in E^{\times p}$.
\end{proposition}

\begin{proof}
The first statement follows directly from the definition of exceptionality and conditions (\ref{eq:normcompatibility}) and (\ref{eq:ai.are.fixed}).  The final statement follows from \cite[Thm.~3]{MSS1} together with the result of Albert \cite{A} which shows $E_n/E_0$ is embeddable if and only if $\xi_p \in N^n_0(E_n^\times)$.  Finally, to show that an exceptional element $a_n$ of an embeddable extension $E/F$ satisfies $N^n_{n-1}(a_n) = a_{n-1}^t e^p$ as above, one applies \cite[Lemma~8]{MSS1}.
\end{proof}

As a final remark, we point out that as operators on $E^\times/E^{\times p}$, we have the identity $$\sum_{i=0}^{p^n-1} \sigma^i \equiv (\sigma-1)^{p^n-1}.$$  Hence we have that $N_0^n(e) \equiv e^{(\sigma-1)^{p^n-1}} \mod{E^{\times p}}$.  More generally, we have that the norm operator $N^i_j$ is given by the action of $(\sigma^{p^j}-1)^{p^{i-j}-1} \equiv (\sigma-1)^{p^i-p^j}$.  For the $\F_p[G_i]$-modules $k_mE_i$, we have the related identity $\iota_j^i\circ N_j^i \equiv (\sigma-1)^{p^i-p^j}$.  We will make frequent use of these identities throughout the remainder of the paper.

\subsection{Module Theory}

Finally, we remind the reader of the essential facts about $\F_p[G]$-modules; with the exception of $E^\times/E^{\times p}$, we shall write our $\F_p[G]$-modules additively.

Much of the theory of $\F_p[G]$-modules comes from the fact that $\F_p[G]$ is a discrete valuation ring with maximal ideal generated by $\sigma-1$, where for the rest of the paper we use $\sigma$ to denote a generator of $G$.  For instance, this tells us that the cyclic submodule generated by an element $w$ is isomorphic to the indecomposable $\F_p[G]$-module $$A_{\ell(w)} := \F_p[G]/(\sigma-1)^{\ell(w)},$$ where $\ell(w)$ --- the so-called length of $w$ --- is defined as the minimum value of $\ell$ so that $(\sigma-1)^\ell w = 0$.  In turn this implies that the $\F_p$-dimension of the cyclic submodule generated by $w$ is $\ell(w)$.  From this it is not difficult to see that $\ell(w+v) \leq \max\{\ell(w),\ell(v)\}$, with equality whenever $\ell(w) \neq \ell(v)$.

We can also prove the following

\begin{lemma}[Exclusion Lemma]\label{le:exclusion}
Suppose that $\{U_\alpha\}_{\alpha \in \mathcal{A}}$ are $\F_p[G]$-submodules of  a fixed $\F_p[G]$-module $W$.  Then the submodules $\{U_\alpha\}$ are $\F_p[G]$-independent if and only if the $\F_p$-subspaces $\{U_\alpha^G\}$ are independent.  Equivalently, $$\sum_{\alpha \in \mathcal{A}} U_\alpha^G = \bigoplus_{\alpha \in \mathcal{A}} U_\alpha^G  \Longleftrightarrow \sum_{\alpha \in \mathcal{A}} U_\alpha =  \bigoplus_{\alpha \in \mathcal{A}} U_\alpha.$$
\end{lemma}
\begin{proof}
We prove the result by induction.  Notice that it suffices to prove the result for a finite collection of submodules, since a dependence amongst an infinite collection of submodules is defined to be a dependence amongst some finite subcollection.

In the case of the two modules $U$ and $V$, the $\F_p[G]$-independence of $U$ and $V$ implies the $\F_p$-independence of $U^G$ and $V^G$.  So suppose that $U^G$ and $V^G$ are $\F_p$-independent, and we show that $U \cap V = \{0\}$. Suppose that $x \in U \cap V$, and note that $(\sigma-1)^{\ell(x)-1}x \in U^G \cap V^G$.  Since $U^G \cap V^G = \{0\}$ by assumption, we have $(\sigma-1)^{\ell(x)-1}x = 0$.  Since $(\sigma-1)^{\ell(w)-1}w \neq 0$ whenever $w \neq 0$, we conclude that $x=0$.

To prove the result for a collection of $m$ submodules $U_1, \cdots, U_m$, notice that by induction the $\F_p$-independence of $\{U_i^G\}_{i=1}^{m-1}$ implies $V = \sum_{i<m}U_i = \oplus_{i<m}U_i$.  From the paragraph above, the $\F_p$-independence of $U_m^G$ and $V^G  = \oplus_{i=1}^{m-1}U_i^G$ then gives $V \cap U_m = \{0\}$, so that $$\sum_{i\leq m}U_i = V + U_m = V \oplus U_m = \left(\oplus_{i<m}U_i\right) \oplus U_m = \bigoplus_{i\leq m}U_i.$$
\end{proof}

Though seemingly humble, this theorem unlocks the structure of an arbitrary $\F_p[G]$-module $W$.  Toward this end, for an $\F_p[G]$-module $W$, we will write $V_i$ for the submodule $\im((\sigma-1)^{i-1}) \cap W^G$.  Notice that each $V_i$ is an $\F_p$-space with trivial $G$-action, and that the $V_i$ provide a filtration of $W^G$:$$V_{p^n} \subseteq  V_{p^n-1} \subseteq  \cdots \subseteq  V_2 \subseteq  V_1 = W^G.$$

\begin{corollary}\label{cor:choose.your.decomposition}
For an $\F_p[G]$-submodule $W$, let $\mathcal{I}_{p^n}$ be a basis for $V_{p^n}$, and for $1 \leq k < p^n$ let $\mathcal{I}_k$ be a basis for a complement of $V_{k+1}$ within $V_{k}$.  For each $x \in \mathcal{I}_k$, let $\alpha_x \in W$ satisfy $(\sigma-1)^{k-1}\alpha_x = x$.  Then
$$W = \bigoplus_{k=1}^{p^n} \bigoplus_{x \in \mathcal{I}_k} \langle \alpha_x \rangle_{\F_p[G]}.$$
\end{corollary}
\begin{proof}
By construction, for each generator $\alpha_x$ we have $\langle \alpha_x \rangle^G = \langle x \rangle$.  Since $\cup_i \mathcal{I}_i$ consists of $\F_p$-independent elements, the Exclusion Lemma shows that \begin{equation}\label{eq:proposed.sum}\sum_k \sum_{x \in \mathcal{I}_k} \langle \alpha_x \rangle_{\F_p[G]} = \bigoplus_k \bigoplus_{x \in \mathcal{I}_k} \langle \alpha_x \rangle_{\F_p[G]}.\end{equation}  To see that this sum is all of $W$, note first that by construction it contains all elements in $W^G$, and hence all elements of length $1$.

Assume that the sum contains all elements of length $\ell$.  For a given element $\gamma$ with $\ell(\gamma) = \ell+1$, we know that $f = (\sigma-1)^{\ell(\gamma)-1}\gamma \in V_{\ell+1}$.  Hence we may write $$f = \sum_{k \geq \ell+1} \sum_{x \in \mathcal{I}_k} c_x x.$$ Then we have $$(\sigma-1)^{\ell}\left(\gamma - \sum_{k \geq \ell+1} \sum_{x \in \mathcal{I}_k} c_x(\sigma-1)^{k-\ell-1}\alpha_x\right) = f -\sum_{k\geq \ell+1}\sum_{x \in \mathcal{I}_k} c_x x = 0.$$  It follows that $\gamma - \sum_{k \geq \ell+1} \sum_{x \in \mathcal{I}_k} c_x(\sigma-1)^{k-\ell-1}\alpha_x$ has length at most $\ell$ and is therefore in (\ref{eq:proposed.sum}) by induction.  Since it is obvious that the term $\sum_{k \geq \ell+1} \sum_{x \in \mathcal{I}_k} c_x(\sigma-1)^{k-\ell-1}\alpha_x$ is in (\ref{eq:proposed.sum}), so too is $\gamma$.
\end{proof}

As a final note on module structures, we point out that the result above can be used to show that all indecomposable $\F_p[G]$-modules are cyclic: non-cyclic modules $W$ can be shown to satisfy $\dim_{\F_p}W^G > 1$, and so the above recipe produces a nontrivial decomposition.  This fact can in turn be used to show that the decomposition of an $\F_p[G]$-module $W$ is essentially unique, in the sense that it determines the indecomposable types which appear in an $\F_p[G]$-decomposition of $W$, together with their multiplicities.  This is recorded in the following

\begin{corollary}\label{cor:general.decomposition}
For an $\F_p[G]$-module $W$, suppose $W = \oplus_{\alpha \in \mathcal{A}} W_\alpha$ where each $W_\alpha$ is indecomposable.  Then $\mathcal{A}$ is a disjoint union of subsets $\mathcal{A}_1, \mathcal{A}_2, \cdots, \mathcal{A}_{p^n}$ where
\begin{itemize}
\item $|\mathcal{A}_i|$ is the codimension of $V_{i+1}$ within $V_i$, and
\item for each $\alpha \in \mathcal{A}_i$ there is an $\F_p[G]$-isomorphism $W_{\alpha}\simeq A_i$.
\end{itemize}
\end{corollary}



\section{The Submodule $\Gamma(m,n)$}\label{sec:Gamma}

As Corollary \ref{cor:general.decomposition} suggests, the driving force in determining a decomposition of $k_mE$ involves understanding the submodule $(k_mE_n)^G \subseteq  k_mE_n$, particularly the filtration \begin{equation*}\begin{split}
V_{p^n} \subseteq  V_{p^n-1} \subseteq  \cdots \subseteq  V_2 \subseteq  V_1 = k_mE_n^G,\end{split}\end{equation*}
where $V_i := \im((\sigma-1)^{i-1})\cap (k_mE_n)^G$.  In the case that $n=1$, the authors of \cite{LMS3} were able to control this filtration by carefully studying the interplay between $\ker(\iota_{E/F})$ and $\im(N_{E/F})$.  In particular, they showed that elements in $\ker(N_{E/F})$ had particularly nice module-theoretic properties, and that ``most'' other elements came from the submodule $\ker(\iota_{E/F} \circ N_{E/F})$.  The challenge was then to construct a submodule $X \subseteq  k_mE$ that was sufficiently ``small'' and so that $N_{E/F}(X) = \ker(\iota_{E/F})$.  This submodule $X$ could then be used to control the module-theoretic properties of other elements in $k_mE$, thus forcing the resulting stratified decomposition.

Our approach will take this same basic strategy, though we will focus much of our attention on the subextension $E_n/E_{n-1}$ and its associated inclusion and norm maps.  We will start by giving an in-depth study of the module structure of $\ker(\iota_{n-1}^n)$.  We point out that the results of this section do not use the embeddability of $E/F$; instead, we only use the fact that $\Gal(E/F) \simeq \Z/p^n\Z$ and that $\xi_p \in E$, where $p>2$ is a prime.

Exact Sequence (\ref{eq:kummer.for.k.theory}) tells us that $$\ker(k_mE_{n-1} \xrightarrow{~\iota^n_{n-1}} k_mE_n) = \{a_{n-1}\}\cdot k_{m-1}E_{n-1},$$ where $a_{n-1}$ has $E_{n-1}(\root{p}\of{a_{n-1}}) = E_n$ and satisfies conditions (\ref{eq:normcompatibility}) and (\ref{eq:ai.are.fixed}).  
Furthermore, we know that \begin{equation}\label{eq:kernel.of.mult.by.a.is.norm}\ker(k_{m-1}E_{n-1} \xrightarrow{\{a_{n-1}\}\cdot-}k_mE_{n-1}) = N^n_{n-1}(k_{m-1}E_n).\end{equation}  Hence to understand $\ker(\iota_{n-1}^n)$, we must find a complement $\Gamma(m,n)$ to $N_{n-1}^n(k_{m-1}E_n)$ in $k_{m-1}E_{n-1}$.  The main result of this section is finding this complement, as recorded in the following

\begin{proposition}\label{prop:Gamma module}
There exists a submodule $\Gamma(m,n) \subseteq  k_{m-1}E_{n-1}$ such that
\begin{enumerate}
\item\label{it:i}  $\Gamma(m,n) = \oplus_{i = 0}^{n-1} \mathcal{Z}_i$ where each $\mathcal{Z}_i \subseteq  \iota^{n-1}_i(k_{m-1}E_i)$ is a direct sum of free $\F_p[G_i]$-modules, and $\mathcal{Z}_i^{G_{n-1}} \subseteq  \iota^{n-1}_0\left(N^i_0\left(k_{m-1}E_i\right)\right)$;
\item\label{it:ii} $\Gamma(m,n) \oplus N^n_{n-1}(k_{m-1}E_n) = k_{m-1}E_{n-1}$;
\item\label{it:iii} $\{a_{n-1}\}\cdot k_{m-1}E_{n-1} = \{a_{n-1}\}\cdot \Gamma(m,n)$; and \item\label{it:iv} as $\F_p[G_{n-1}]$-modules, $\Gamma(m,n)\simeq \{a_{n-1}\}\cdot \Gamma(m,n)$ under the map $\gamma \mapsto \{a_{n-1}\}\cdot \gamma$.
\end{enumerate}
\end{proposition}

A few remarks are in order.  First, the uniqueness of an $\F_p[G_{n-1}]$-decomposition of $k_{m-1}E_{n-1}$ implies that all complements to $N_{n-1}^n(k_{m-1}E)$ are isomorphic as $\F_p[G_{n-1}]$-modules, provided that a complement to $N_{n-1}^n(k_{m-1}E_n)$ exists.  Furthermore, properties (\ref{it:iii}) and (\ref{it:iv}) follow directly from property (\ref{it:ii}): the former because of Equation (\ref{eq:kernel.of.mult.by.a.is.norm}), and the latter because $a_{n-1}$ has trivial $G_{n-1}$-action (condition (\ref{eq:ai.are.fixed})).  Hence the content of this theorem is in showing that $N_{n-1}^n(k_{m-1}E_n)$ is a summand of $k_{m-1}E_{n-1}$, and that this latter module is appropriately stratified.




To prove this result, we shall use induction.  For our base cases, suppose that either $n=1$ or $m=1$.  In either case $N_{n-1}^n(k_{m-1}E_n)$ is a submodule of the trivial $\F_p[G_{n-1}]$-module $k_{m-1}E_{n-1}$, and hence has a complement which is also trivial as an $\F_p[G_{n-1}]$-module.  Such a complement obviously satisfies condition \ref{it:i}. 

Suppose then that $m,n>1$, and assume by induction the existence of a submodule $\Gamma(m-1,n) \subseteq  k_{m-2}E_{n-1}$ which satisfies the conclusions of Proposition \ref{prop:Gamma module}.

\begin{lemma}\label{le:preliminaryformathcalK}
For $\gamma \in \iota^{n-1}_0(k_{m-1}E_0)$ with $\iota^n_{n-1}(\gamma) = 0$, there exists $\alpha \in k_{m-1}E_{n-1}$ so that $\iota^{n-1}_{0}(N^{n-1}_0(\alpha)) = \gamma$ and $\iota^n_{n-1}(\alpha) = 0$.
\end{lemma}

\begin{proof}
By Exact Sequence (\ref{eq:kummer.for.k.theory}) we have $\gamma = \{a_{n-1}\}\cdot g$ for some $g \in k_{m-2}E_{n-1}$.  Proposition \ref{prop:Gamma module} says that we may take $$g \in \Gamma(m-1,n)^{G_{n-1}} \subseteq  \iota_0^{n-1}(k_{m-2}E_0) \subseteq  \iota_{n-2}^{n-1}(k_{m-2}E_{n-2}),$$ say $g = \iota_{n-2}^{n-1}(\hat g)$.  
We now compute $N_{n-2}^{n-1}(\gamma)$ in two ways.  On the one hand, since $\gamma \in \im(\iota_{0}^{n-1})$ we have $N_{n-2}^{n-1}(\gamma) = 0$.  On the other hand, since $\gamma = \{a_{n-1}\}\cdot \iota_{n-2}^{n-1}(\hat g)$, the Projection Formula (\ref{eq:projection.formula}) gives $$0=N^{n-1}_{n-2} (\gamma) = N^{n-1}_{n-2} ( \{a_{n-1}\}\cdot g) = \{a_{n-2}\}\cdot \hat g.$$

By Exact Sequence (\ref{eq:kummer.for.k.theory}) we conclude that $\hat g \in N^{n-1}_{n-2} \left(k_{m-2}E_{n-1}\right)$, and therefore $\iota_{n-2}^{n-1}(\hat g)$ is in the image of $(\sigma-1)^{p^{n-1}-p^{n-2}}$.  This shows $g$ lies in the fixed part of a submodule of $k_{m-2}E_{n-1}$ of length at least $p^{n-1}-p^{n-2} + 1 > p^{n-2}$.  Since by induction $\Gamma(m-1,n)$ is a direct sum of free $\F_p[G_{i}]$-submodules for $0 \leq i \leq n-1$, $g =(\sigma-1)^{p^{n-1}-1}(\alpha')$ for some $\alpha' \in k_{m-2}E_{n-1}$.  Letting $\alpha = \{a_{n-1}\}\cdot \alpha'$ we have $\iota^n_{n-1}(\alpha) = 0$ and
\begin{equation*}\begin{split}
\iota^{n-1}_0(N^{n-1}_0(\alpha)) &= (\sigma-1)^{p^{n-1}-1} (\alpha) = (\sigma-1)^{p^{n-1}-1}(\{a_{n-1}\}\cdot \alpha')) \\&=\{a_{n-1}\}\cdot (\sigma-1)^{p^{n-1}-1}(\alpha') = \{a_{n-1}\}\cdot g  = \gamma,
\end{split}\end{equation*} as desired.
\end{proof}

\begin{lemma}\label{le:choosing.decomposition.wisely}
There exists a module decomposition $$k_{m-1} E_{n-1} = \mathcal{X}_0 \oplus \cdots \oplus \mathcal{X}_{n-2} \oplus \mathcal{Y}_0 \oplus \cdots \oplus \mathcal{Y}_{n-1}$$ satisfying the conditions of
Theorem \ref{th:embeddable.decomposition}, and with the properties
\begin{itemize}
\item $\mathcal{X}_i \subseteq  \{a_{n-1}\}\cdot k_{m-2}E_{n-1}$ for each $i$, and
\item $\mathcal{Y}_{n-1} = \mathcal{K} \oplus \mathcal{N} \oplus \hat{\mathcal{Y}}_{n-1}$, where each of these submodules is free over $\F_p[G_{n-1}]$, and so that
\begin{enumerate}
\item $\mathcal{K} \subseteq  \ker \iota^n_{n-1}$ and
\item $\mathcal{N} \subseteq  N^n_{n-1}(k_{m-1}E_n)$.
\end{enumerate}
\end{itemize}
\end{lemma}
\begin{proof}
We shall let our decomposition come from an arbitrary decomposition $\mathcal{X}_0 \oplus \cdots \oplus \mathcal{X}_{n-2} \oplus \mathcal{Y}_0\oplus \cdots \oplus \mathcal{Y}_{n-1}$ of $k_{m-1}E_{n-1}$ provided by induction, subject to a few conditions on $\mathcal{X}$ and $\mathcal{Y}$ we are free to impose. First, Proposition \ref{prop:exceptional.element.criteria} gives that $a_{n-1}$ is an exceptional element for the extension $E_{n-1}/F$,  and so Theorem \ref{th:embeddable.decomposition} tells us that the decomposition can be chosen so that $\mathcal{X}_i \subseteq  \{a_{n-1}\}\cdot \iota^{n-1}_i(k_{m-2}E_i) \subseteq  \{a_{n-1}\}\cdot k_{m-2}E_{n-1}$.

Second, Corollary \ref{cor:choose.your.decomposition} gives us a great deal of freedom in choosing the submodule $\mathcal{Y}_{n-1}$. Specifically, since $\iota^{n-1}_0\circ N^{n-1}_0$ is given by the action of $(\sigma-1)^{p^{n-1}-1}$, we may choose any $\F_p$-basis $\mathcal{I}$ of $\iota^{n-1}_0(N^{n-1}_0(k_{m-1}E_{n-1}))$ and --- for every $x \in \mathcal{I}$ --- an element $\alpha_x \in k_{m-1}E_{n-1}$ so that $\iota^{n-1}_0(N^{n-1}_0(\alpha_x)) = x$.  Then Corollary \ref{cor:choose.your.decomposition} says that $\mathcal{Y}_{n-1}$ can be taken to be $\oplus_{x \in \mathcal{I}} \langle \alpha_x \rangle_{\F_p[G_{n-1}]}$.

We choose our basis $\mathcal{I}$ as the disjoint union of $\mathcal{I}_K, \mathcal{I}_N$ and $\hat{\mathcal{I}}$, where
\begin{enumerate}
\item $\mathcal{I}_K$ is a basis for $\ker \iota^n_{n-1} \cap \iota^{n-1}_0(N^{n-1}_0(k_{m-1}E_{n-1}))$;
\item $\mathcal{I}_N$ is a basis for a complement to $$\ker\iota^n_{n-1} \cap \iota^{n-1}_0(N^n_0(k_{m-1}E_n)) \quad \mbox{in} \quad \iota^{n-1}_0(N^n_0(k_{m-1}E_n));$$
\item and $\hat{\mathcal{I}}$ is a basis for a complement to $$\langle \mathcal{I}_K,\mathcal{I}_N\rangle_{\F_p} \quad \mbox{in} \quad \iota^{n-1}_0(N^{n-1}_0(k_{m-1}E_{n-1})).$$
\end{enumerate}

By Lemma \ref{le:preliminaryformathcalK}, for every $x \in \mathcal{I}_K$ there exists $\alpha_x$ so that $\iota^{n-1}_0(N^{n-1}_0(\alpha_x)) = x$ and $\alpha_x \in \ker\iota^n_{n-1}$. Hence we define $$\mathcal{K}:=\oplus_{x \in \mathcal{I}_K} \langle \alpha_x \rangle_{\F_p[G_{n-1}]} \subseteq  \ker\iota_{n-1}^n.$$ 

For each $x \in \mathcal{I}_N$, there exists $\beta \in k_{m-1}E_n$ so that $\iota^{n-1}_0(N^n_0(\beta)) = x$, and therefore $\iota^{n-1}_0(N^{n-1}_0(N^n_{n-1}(\beta))) = x$.  Hence we define $$\mathcal{N} :=\oplus_{x \in \mathcal{I}_N} \langle N^n_{n-1}(\beta)\rangle_{\F_p[G_{n-1}]} \subseteq  N^n_{n-1}(k_{m-1}E_n).$$  

For each $x \in \hat{\mathcal{I}}$ we choose arbitrary $\alpha_x \in k_{m-1}E_{n-1}$ to satisfy $\iota^{n-1}_0(N^{n-1}_0(\alpha_x)) = x$, and we let $$\hat{\mathcal{Y}}_{n-1} := \oplus_{x \in \hat{\mathcal{I}}}\langle \alpha_x \rangle_{\F_p[G_{n-1}]}.$$
%
\end{proof}

We will show that the submodule $\Gamma(m,n)$ of Proposition \ref{prop:Gamma module} is $\mathcal{Y}_0 \oplus \cdots \oplus \mathcal{Y}_{n-2} \oplus \hat{\mathcal{Y}}_{n-1}$.  We proceed by determining a complement for $N^n_{n-1}(k_{m-1}E_n)$ in $k_{m-1}E_{n-1}$, beginning with a calculation of $\ker(\iota^n_{n-1})$.

\begin{lemma}\label{le:ker}
Using the notation from Lemma \ref{le:choosing.decomposition.wisely}, $$\ker\left(\xymatrix{k_{m-1}E_{n-1} \ar[r]^{\iota^n_{n-1}} & k_{m-1}E_n}\right) =
\mathcal{X}_0 \oplus \cdots \oplus \mathcal{X}_{n-2} \oplus \mathcal{K}.$$
\end{lemma}

\begin{proof}
Since $\mathcal{X}_i \subseteq  \{a_{n-1}\}\cdot k_{m-2}E_{n-1}$, Exact Sequence (\ref{eq:kummer.for.k.theory}) gives $\mathcal{X}_i \subseteq  \ker \iota^n_{n-1}$.  Lemma \ref{le:choosing.decomposition.wisely} also gives $\mathcal{K} \subseteq  \ker \iota^n_{n-1}$.   We complete the proof by showing that
\begin{equation}\label{eq:verifying.the.kernel.of.inclusion}
\begin{split}
\ker(\iota^n_{n-1}) \cap \left(\mathcal{Y}_0 \oplus \cdots \oplus \mathcal{Y}_{n-2} \oplus \hat{\mathcal{Y}}_{n-1} \oplus \mathcal{N}\right) = \{0\}.
\end{split}
\end{equation}
To do this we show that the fixed submodule of the direct sum above has trivial intersection with $\ker \iota^n_{n-1}$ (after which we can appeal to the Exclusion Lemma (\ref{le:exclusion})).

Since $\mathcal{N}, \hat{\mathcal{Y}}_{n-1} \subseteq  \mathcal{Y}_{n-1}$, Theorem \ref{th:embeddable.decomposition} gives  $$\ker \iota^n_{n-1} \cap \left(\mathcal{Y}_0 \oplus \cdots \oplus \mathcal{Y}_{n-2} \oplus \hat{\mathcal{Y}}_{n-1} \oplus \mathcal{N}\right)^{G_{n-1}} \subseteq  \ker \iota^n_{n-1} \cap \iota^{n-1}_0(k_{m-1}E_0).$$  Lemma \ref{le:preliminaryformathcalK}, on the other hand, shows that $$\ker \iota^n_{n-1} \cap \im~\iota^{n-1}_0 \subseteq  \ker \iota^n_{n-1} \cap \iota^{n-1}_0(N^{n-1}_0(k_{m-1}E_{n-1})) = \langle \mathcal{I}_K \rangle = \mathcal{K}^{G_{n-1}}.$$  Since the fixed parts of each of the modules $\mathcal{Y}_i$ ($0 \leq i \leq n-2$), $\hat{\mathcal{Y}}_{n-1}$ and $\mathcal{N}$ are $\F_p$-independent from the fixed part of $\mathcal{K}$, the Exclusion Lemma (\ref{le:exclusion}) implies that Equation (\ref{eq:verifying.the.kernel.of.inclusion}) is true.
\end{proof}

\begin{lemma}
Using the notation from Lemma \ref{le:choosing.decomposition.wisely}, $$\im\left(\xymatrix{k_{m-1}E_n \ar[r]^{N^n_{n-1}~} & k_{m-1}E_{n-1}}\right) = \mathcal{X}_0 \oplus \cdots \oplus \mathcal{X}_{n-2} \oplus \mathcal{K} \oplus \mathcal{N}.$$
\end{lemma}

\begin{proof}
Let $a_n$ be an exceptional element of $E/F$, and choose $t$ so that $N^n_{n-1}(a_n^t) \in a_{n-1}E_{n-1}^{\times p}$.  An element $\gamma \in \ker \iota^n_{n-1}$ takes the form $\gamma = \{a_{n-1}\}\cdot g$ by Exact Sequence (\ref{eq:kummer.for.k.theory}), and so the Projection Formula (\ref{eq:projection.formula}) gives $N^n_{n-1}(\{a_n^t\}\cdot\iota^n_{n-1}(g)) = \{a_{n-1}\}\cdot g$.  Hence by Lemma \ref{le:ker}, $$\ker \iota^n_{n-1} = \mathcal{X}_0 \oplus \cdots \oplus \mathcal{X}_{n-2} \oplus \mathcal{K} \subseteq  N^n_{n-1}(k_{m-1}E_n).$$  Of course $\mathcal{N}$ is constructed so that $\mathcal{N} \subseteq  N^n_{n-1}(k_{m-1}E_n)$, and so we have $$\mathcal{X}_0 \oplus \cdots \oplus \mathcal{X}_{n-2} \oplus \mathcal{K} \oplus \mathcal{N} \subseteq N^n_{n-1}(k_{m-1}E_n).$$

For the opposite containment, it is enough to show $$N^n_{n-1}(k_{m-1}E_n) \cap \left(\mathcal{Y}_0 \oplus \cdots \oplus \mathcal{Y}_{n-2} \oplus \hat{\mathcal{Y}}_{n-1}\right) = \{0\}.$$ By the Exclusion Lemma (\ref{le:exclusion}), this is equivalent to showing the associated fixed submodules are $\F_p$-independent.  We will verify this by showing $$(N^n_{n-1}(k_{m-1}E_n))^{G_{n-1}} \subseteq  \mathcal{X}_0 \oplus \cdots \oplus \mathcal{X}_{n-2} \oplus \mathcal{K} \oplus \mathcal{N}.$$ 

Let $\gamma$ be an element in $(N^n_{n-1}(k_{m-1}E_n))^{G_{n-1}}$, say $\gamma = N^n_{n-1}(\alpha)$ for some $\alpha \in k_{m-1}E_n$.  If $\iota^n_{n-1}(\gamma) = 0$ then $\gamma \in \ker\iota^n_{n-1} = \mathcal{X}_0 \oplus \cdots \oplus \mathcal{X}_{n-2} \oplus \mathcal{K}$ by Lemma~\ref{le:ker}, and we are done.  Otherwise $\gamma \notin \ker \iota^n_{n-1}$, and so $\iota^n_{n-1}(\gamma) = \iota^n_{n-1}(N^n_{n-1}(\alpha)) \neq 0$.  Since $\iota^n_{n-1} \circ N^n_{n-1}$ is represented by the polynomial $$\sigma^{p^{n-1}} + \cdots + \sigma^{p^{n-1}(p-1)} \equiv (\sigma-1)^{p^n-p^{n-1}},$$ this implies that $\ell(\alpha) > p^n-p^{n-1} \geq 2p^{n-1}$. The decomposition of $k_{m-1}E_n$ provided by Theorem \ref{th:coarse.decomposition} implies $\iota^n_{n-1}(\gamma)$ is the fixed part of a submodule of dimension $p^n$; i.e., $\iota^n_{n-1}(\gamma) = \iota^n_0(N^n_0(\beta))$ for some $\beta \in k_{m-1}E_n$.  If we let $\delta = \iota_0^{n-1} (N_0^n(\beta))$, then we have $\iota^n_{n-1}(\gamma - \delta) = 0$.  Hence we have $\gamma - \delta \in \ker(\iota^n_{n-1})$, from which it follows that $$\gamma \in \iota^{n-1}_0(N^n_0(k_{m-1}E_n)) + \ker \iota^n_{n-1}.$$


Recall, however, that $\mathcal{N}^{G_{n-1}} = \langle \mathcal{I}_N \rangle_{\F_p}$ was chosen as a complement to $\ker \iota^n_{n-1} \cap \iota^{n-1}_0(N^n_0(k_{m-1}E_n)) \subseteq  \langle \mathcal{I}_K \rangle_{\F_p}$ in $\iota^{n-1}_0(N^n_0(k_{m-1}E_n))$.  Hence we have $\langle\mathcal{I}_K,\mathcal{I}_N\rangle_{\F_p} \supseteq \iota^{n-1}_0(N^n_0(k_{m-1}E_n))$, and so $$\gamma \in \langle\mathcal{I}_K,\mathcal{I}_N\rangle_{\F_p} + \ker \iota^n_{n-1} \subseteq  \mathcal{X}_0 \oplus \cdots \oplus \mathcal{X}_{n-2} \oplus \mathcal{K} \oplus \mathcal{N}.$$
\end{proof}

\begin{proof}[Proof of Proposition \ref{prop:Gamma module}]
For each $0 \leq i < n-1$ define $\mathcal{Z}_i := \mathcal{Y}_i$, and define $\mathcal{Z}_{n-1} := \hat{\mathcal{Y}}_{n-1}$.  We define $\Gamma(m,n) := \mathcal{Z}_{0} \oplus \cdots \oplus \mathcal{Z}_{n-1}$.  The previous lemmas show that $\Gamma(m,n)$ satisfies (\ref{it:i}) and (\ref{it:ii}), and we have already verified that properties (\ref{it:iii}) and (\ref{it:iv}) follow from (\ref{it:ii}).
\end{proof}

We record the following corollary, since it will be useful later.

\begin{corollary}\label{co:longer than you think}
If $g \in \Gamma(m,n)^{G_{n-1}}$ and $N^{n-1}_{n-2}(\{a_{n-1}\} \cdot g) = 0$, then for some $\alpha \in \Gamma(m,n)$ we have $g = \iota^{n-1}_0(N^{n-1}_0(\alpha))$.
\end{corollary}

\begin{proof}
Since $\Gamma(m,n)^{G_{n-1}} \subseteq  \iota^{n-1}_0 \left(k_{m-1}E_0\right)$, it follows that $g = \iota^{n-1}_{n-2} (\hat g)$ for some $\hat g \in k_{m-1}E_{n-2}$. By the Projection Formula (\ref{eq:projection.formula}) we therefore have $$0 = N^{n-1}_{n-2}(\{a_{n-1}\}\cdot g) = \{a_{n-2}\}\cdot \hat g,$$ and so Exact Sequence (\ref{eq:kummer.for.k.theory}) implies $\hat g = N^{n-1}_{n-2} (\alpha')$ for some $\alpha' \in k_{m-1}E_{n-1}$.  Hence we have $$g = \iota^{n-1}_{n-2}(\hat g) = \iota^{n-1}_{n-2}\left(N^{n-1}_{n-2}(\alpha')\right) = (\sigma-1)^{p^{n-1}-p^{n-2}}\alpha' \in \im(\sigma-1)^{p^{n-1}-p^{n-2}}.$$
Since $\Gamma(m,n)$ is a direct sum of cyclic submodules of dimensions $p^i$ for $0 \leq i \leq n-1$, we must have $g \in \im(\sigma-1)^{p^{n-1}-1}$.  Hence $g \in \mathcal{Z}_{n-1}^{G_{n-1}}$.
\end{proof}

\section{Fixed Elements and Norms}\label{sec:fixed.elements.are.norms}

The key result of this section is Proposition \ref{prop:fixedeltsarenorms}.  This result uses Hilbert 90-like results and facts about abstract $\F_p[G]$-modules to prove that elements in $\ker(N^n_{n-1})$ have ``nice'' module-theoretic properties. Again, we will not assume in this section that the given extension $E/F$ is embeddable --- we will just use the facts that $\Gal(E/F) \simeq \Z/p^n\Z$, that $\xi_p \in E$ and that $p>2$ is prime.

In our setting we need to be careful about the possible difference in length between the $\F_p[G_i]$-submodule generated by an element $\gamma \in k_mE_i$ and the $\F_p[G_n]$-submodule of $k_mE_n$ generated by $\iota^n_{i}(\gamma)$.  Towards this end, we give results for determining when an element lies in the submodule $\im(\iota^n_i)$ and --- when it does --- for controlling the $\F_p[G_i]$-lengths of representatives from $k_mE_i$ for this element.

We also establish notation to distinguish these potentially different notions of length: for an element $\gamma \in k_mE_i$, we write $\ell_{G_i}(\gamma)$ to denote the length of the $\F_p[G_i]$-submodule generated by $\gamma$.  In the same way, the $\F_p$-dimension of the $\F_p[H_i]$-submodule generated by $\gamma \in k_mE_n$ is denoted $\ell_{H_i}(\gamma)$.  Since $H_i = \langle \sigma^{p^i}\rangle$, we note that
$$\ell_{H_i}(\gamma) = \min\{\ell: (\sigma^{p^i}-1)^\ell\gamma = 0\} = \min\{\ell:(\sigma-1)^{p^i\ell} \gamma = 0\}.$$

\begin{lemma}\label{le:warmup to fixed base}
If $N^n_{n-1}(\gamma) = 0$ and $\gamma \in (k_mE_n)^{H_{n-1}}$, then
there exists $\hat \gamma \in k_mE_{n-1}$ such that $\iota^n_{n-1} (\hat
\gamma) = \gamma$ and $\ell_{G_{n-1}}(\hat \gamma) = \ell_G(\gamma)$.
Additionally, if $\ell_G(\gamma) \leq p^{n-1}-p^{n-2}$ we may insist
$N^{n-1}_{n-2}(\hat \gamma) =0$.
\end{lemma}

\begin{proof}
\cite[Lemma 3]{LMS3} shows that the sequence $$\xymatrix{k_mE_{n-1} \ar[r]^-{\iota_{n-1}^n} &(k_mE_n)^{H_{n-1}} \ar[r]^-{N^n_{n-1}}& \{a_{n-1}\}\cdot k_{m-1}E_{n-1}}$$ is exact, so for $N_{n-1}^n(\gamma) = 0$ we may conclude $\gamma = \iota_{n-1}^n(\hat \gamma)$ for some $\hat \gamma \in k_mE_{n-1}$.  Notice also that when $n = 1$ the length condition is trivial, so we may assume that $n \geq 2$.

We now argue that $\hat \gamma$ may be taken so that
$\ell_{G_{n-1}}(\hat \gamma) = \ell_G(\gamma)$.  We cannot have
$\ell_{G_{n-1}}(\hat \gamma) < \ell_G(\gamma)$, since if
$(\sigma-1)^x \hat \gamma = 0 \in k_{m}E_{n-1}$ then $$(\sigma-1)^x \gamma=(\sigma-1)^x \iota^n_{n-1}(\hat \gamma)= \iota^n_{n-1}\left((\sigma-1)^x \hat \gamma\right) = 0 .$$

So suppose that $\ell:=\ell_{G_{n-1}}(\hat \gamma) > \ell_G(\gamma)$.  Our goal is to use Corollary \ref{co:longer than you think} to adjust $\hat \gamma$ by an element $\{a_{n-1}\}\cdot \alpha \in k_{m}E_{n-1}$ in order to produce an element of smaller length whose image under inclusion is $\gamma$. For this we study $f:=(\sigma-1)^{\ell-1}\hat \gamma$.

First, by induction we know $k_m E_{n-1} = \mathcal{X}_0 \oplus \cdots \oplus \mathcal{X}_{n-2} \oplus
\mathcal{Y}_0 \oplus \cdots \oplus \mathcal{Y}_{n-1}$, where by Theorem \ref{th:embeddable.decomposition} we have $\mathcal{X}_i \subseteq  \{a_{n-1}\}\cdot \iota^{n-1}_i(k_{m-1}E_{i}) \subseteq
\ker \iota_{n-1}^n$.  Hence we may take $\hat \gamma \in
\mathcal{Y}_0 \oplus \cdots \oplus \mathcal{Y}_{n-1}$.  Since $f :=(\sigma-1)^{\ell-1}\hat \gamma \in \left(\mathcal{Y}_0 \oplus \cdots \oplus \mathcal{Y}_{n-1}\right)^{G_{n-1}}$ we have $f \in \iota_0^{n-1} \left(k_m E_0\right)$.  Since $n \geq 2$, we therefore conclude \begin{equation}\label{eq:norm of f}N^{n-1}_{n-2}(f) = 0.\end{equation}

On the other hand, since $\ell>\ell(\gamma)$ we know $f \in \ker(\iota_{n-1}^n)$.  Exact Sequence (\ref{eq:kummer.for.k.theory}) and Proposition \ref{prop:Gamma module} then imply $f \in \{a_{n-1}\}\cdot \Gamma(n,m)^{G_{n-1}}$, say $f = \{a_{n-1}\}\cdot \iota_0^{n-1}(g)$. Recalling Equation (\ref{eq:norm of f}), the Projection Formula (\ref{eq:projection.formula}) gives $$0=N^{n-1}_{n-2}(f) = N^{n-1}_{n-2} \left(\{a_{n-1}\}\cdot \iota^{n-1}_0(g)\right) = \{a_{n-2}\}\cdot \iota^{n-2}_0 (g).$$
This allows us to apply Corollary \ref{co:longer than you think}, and we conclude that $\iota_0^{n-1}(g) = \iota_0^{n-1}(N^{n-1}_0(\alpha))$ for some $\alpha \in \Gamma(m,n)$.  Since $\ell_{G_{n-1}}(\{a_{n-1}\}\cdot\alpha) = p^{n-1}$  by Proposition \ref{prop:Gamma module}(\ref{it:iv}) and $\iota_{n-1}^n(\{a_{n-1}\}\cdot \alpha) = 0$, we see that $$\hat \gamma - (\sigma-1)^{p^{n-1}-\ell_{G_{n-1}}(\hat \gamma)}(\{a_{n-1}\}\cdot \alpha)$$ has $G_{n-1}$-length smaller than $\ell_{G_{n-1}}(\hat \gamma)$ and has image $\gamma$ under $\iota_{n-1}^n$.
We iterate this process until we have constructed an element $\hat \gamma$ so that $\iota^n_{n-1}(\hat \gamma) = \gamma$ and $\ell_{G_{n-1}}(\hat \gamma) = \ell_G(\gamma)$.

All we have left is to show that if $\ell_G(\gamma) \leq p^{n-1} - p^{n-2}$, then we may insist $N^{n-1}_{n-2}(\hat \gamma) = 0$.  For this, since $\ell_{G_{n-1}}(\hat \gamma) \leq p^{n-1}-p^{n-2}$ we have $$(\sigma-1)^{p^{n-1}-p^{n-2}}(\hat \gamma) = \iota_{n-2}^{n-1} (N^{n-1}_{n-2}(\hat \gamma)) = 0,$$ so $N^{n-1}_{n-2}(\hat \gamma) = \{a_{n-2}\}\cdot g$ for some $g \in \Gamma(m,n-1) \subseteq  k_{m-1}E_{n-2}$ by Proposition \ref{prop:Gamma module}.  We claim that $$\hat \gamma' := \hat \gamma - \{a_{n-1}\}\cdot \iota^{n-1}_{n-2}(g)$$ has the desired inclusion, norm and length properties.

To prove this claim, notice first that $\iota_{n-1}^n\left(\{a_{n-1}\}\cdot \iota^{n-1}_{n-2}(g)\right) = 0$ by Exact Sequence (\ref{eq:kummer.for.k.theory}), and hence $\iota^n_{n-1}(\hat \gamma') = \gamma$.  It is also obvious that $N^{n-1}_{n-2}\left(\{a_{n-1}\}\cdot \iota^{n-1}_{n-2}(g)\right) = \{a_{n-2}\}\cdot g$ by the Projection Formula (\ref{eq:projection.formula}), and hence $N^{n-1}_{n-2}(\hat \gamma') = 0$.
For the length condition, notice first that $\ell_{G_{n-1}}(\{a_{n-1}\}\cdot \iota^{n-1}_{n-2}(g)) = \ell_{G_{n-2}}(g)$ by Proposition \ref{prop:Gamma module}(\ref{it:iv}) applied to $\Gamma(m,n-1)$. In view of the properties of length, together with the fact that a preimage of $\gamma$ under $\iota_{n-1}^n$ cannot have $G_{n-1}$-length less than $\ell:=\ell(\gamma) = \ell_{G_{n-1}}(\hat \gamma)$, it will be enough to prove that $\ell \geq \ell_{G_{n-2}}(g)$.  To see that this is true, note that we have $$0 = N_{n-2}^{n-1}\left((\sigma-1)^{\ell}\hat \gamma\right) = (\sigma-1)^{\ell}\left(\{a_{n-2}\}\cdot g\right) = \{a_{n-2}\}\cdot (\sigma-1)^{\ell}g.$$  Applying Proposition \ref{prop:Gamma module}(\ref{it:iv}) again, we have the desired inequality.
\end{proof}

The previous result gives us the fixed submodule under one particular subgroup of $G$.  To find the fixed submodule for the remaining subgroups of $G$, we have the following

\begin{lemma}\label{le:warmup to fixed}
If $N^n_{n-1}(\gamma) = 0$ and $\gamma \in (k_mE_n)^{H_i}$, then there exists $\hat \gamma \in k_mE_i$ such that $\iota^n_i (\hat \gamma) = \gamma$ and $\ell_{G_i}(\hat \gamma) = \ell_G(\gamma)$. Additionally, if $\ell_G(\gamma) \leq p^{i}-p^{i-1}$ we may insist $N^i_{i-1}(\hat \gamma) =0$.
\end{lemma}

\begin{proof}
The base case of this result is the previous lemma.

For the inductive step, let $\gamma \in (k_mE_n)^{H_i}$ with $N^n_{n-1}(\gamma) = 0$, and suppose we have the result for $i+1$. Since $(k_mE_n)^{H_{i}} \subseteq (k_mE_n)^{H_{i+1}}$, there exists $\tilde \gamma \in k_mE_{i+1}$ such that $\iota_{i+1}^n(\tilde \gamma) = \gamma$ and $\ell_{G_{i+1}}(\tilde \gamma) = \ell_G(\gamma)$. Furthermore, since $\ell_G(\gamma) \leq p^i \leq p^{i+1}-p^i$ we may insist $N^{i+1}_i(\tilde \gamma) = 0$.  Applying the previous Lemma to the extension $E_{i+1}/E_i$, this implies that there exists $\hat \gamma \in k_mE_i$ such that $\ell_{G_i}(\hat \gamma) = \ell_{G_{i+1}}(\tilde \gamma)$, $\iota_i^{i+1}(\hat \gamma) = \tilde \gamma$, and so that if $\ell_{G_i}(\hat \gamma) \leq p^i-p^{i-1}$ then we may assume $N^i_{i-1}(\hat \gamma) = 0$.  But then we also have $\ell_{G_{i}}(\hat \gamma) = \ell_G(\gamma)$ and $\iota_i^n(\hat \gamma) = \gamma$ as desired.
\end{proof}

We are now ready for the main result of the section.  We shall state it in some generality and then restrict ourselves to a special case in the subsequent corollary.

\begin{proposition}\label{prop:fixedeltsarenorms}
For $\gamma \in k_mE_n$, if
\begin{itemize}
\item $\ell_{H_j}(\gamma)>2p^{n-j-1}$; or if
\item $E_n/E_j$ is embeddable and $\ell_{H_j}(\gamma)>p^{n-j-1}$; or if
\item $N_{n-1}^n(\gamma) = 0$ and $\ell_{H_j}(\gamma)>p^{n-j-1}$,
\end{itemize}
then $(\sigma^{p^j}-1)^{\ell_{H_j}(\gamma)-1}\gamma \in \iota_j^n(N^n_j (k_mE_{n}))$.
\end{proposition}

\begin{proof}
To prove the claim we proceed by induction on $j$.  The base case is $j=n-1$.  \cite[Lemma 2]{LMS3} verifies that $\ell_{H_{n-1}}(\gamma)>2$ gives the desired conclusion, and additionally shows that $$\im(\sigma^{p^{n-1}}-1) \cap (k_mE_n)^{H_{n-1}} = \iota^n_{n-1}\left(\{\xi_p\}\cdot k_{m-1}E_{n-1}\right) + \iota_{n-1}^n(N_{n-1}^n(k_mE_n)).$$

So suppose that $\ell_{H_{n-1}}(\gamma) = 2$.  In the case that $E_n/E_{n-1}$ is embeddable, Albert's Theorem \cite{A} shows that $\xi_p \in N_{n-1}^n(E_n^\times)$.  The Projection Formula (\ref{eq:projection.formula}) then gives $$\iota_{n-1}^n\left(\{\xi_p\}\cdot k_{m-1}E_{n-1}\right) \subseteq \iota_{n-1}^n(N_{n-1}^n(k_mE_n)).$$ Hence if $E_n/E_{n-1}$ is embeddable and $\ell_{H_{n-1}}(\gamma) = 2$, we are done.

We have left to consider the case where $\ell_{H_{n-1}}(\gamma) = 2$ and $N_{n-1}^n(\gamma) = 0$.  Considering this equation in $K_mE$, we have $N_{n-1}^n(\tilde \gamma) = p\tilde f$ for some $\tilde f \in K_mE_{n-1}$ and preimage $\tilde \gamma \in K_mE_n$ of $\gamma \in k_mE_n$.  Hence we have $N_{n-1}^n(\tilde \gamma-\tilde f) = 0$ as elements of $K_mE_{n-1}$, and so Hilbert 90 for $K$-theory (\ref{prop:hilbert.90.for.k.theory}) implies that there exists $\tilde \alpha \in K_mE_n$ with \begin{equation}\label{eq:hilbert.90.for.fixed.elements.are.norms}\tilde \gamma - \tilde f = (\sigma^{p^{n-1}}-1)\tilde \alpha.\end{equation}  Considering that $\ell_{H_{n-1}}(\gamma) = 2$, we can apply $(\sigma^{p^{n-1}}-1)$ to Equation (\ref{eq:hilbert.90.for.fixed.elements.are.norms}) to give
\begin{equation*}\begin{split}
(\sigma^{p^{n-1}}-1)\tilde \gamma & = (\sigma^{p^{n-1}}-1)\left((\sigma^{p^{n-1}}-1)\tilde \alpha + \tilde f\right) \\&= (\sigma^{p^{n-1}}-1)^2 \tilde \alpha 
\end{split}\end{equation*} The element $\alpha \in k_mE_n$ represented by $\tilde \alpha$ therefore has $\ell_{H_{n-1}}(\alpha) = 3$, and so we appeal to the initial case to show $$(\sigma^{p^{n-1}}-1)^2 \alpha = (\sigma^{p^{n-1}}-1)\gamma \in \iota^n_{n-1}(N^n_{n-1}(k_mE_n)),$$ as desired.

Having settled the base case, we have also completed the case $n=1$.  Now suppose that $n \geq 2$ and the result holds for $j+1$, and we show it also holds for $j$.  For simplicity we let $\varepsilon = 1$ if either $E_n/E_j$ is embeddable or $N^n_{n-1}(\gamma)=0$, and let $\varepsilon = 2$ if both $N^n_{n-1}(\gamma) \neq 0$ and $E_n/E_j$ is not embeddable.  Let $\gamma$ be an arbitrary element with $\ell_{H_j}(\gamma) > \varepsilon p^{n-j-1}$, and consider the element
$$\delta:=(\sigma^{p^j}-1)^{\ell_{H_j}(\gamma)-\varepsilon p^{n-j-1}-1}\gamma.$$ It is easy to see that $\ell_{H_j}(\delta) = \varepsilon p^{n-j-1}+1$ and that $$(\sigma^{p^j}-1)^{\varepsilon p^{n-j-1}}\delta = (\sigma^{p^j}-1)^{\ell_{H_j}(\gamma)-1}\gamma.$$  Hence if we can show $(\sigma^{p^j}-1)^{\varepsilon p^{n-j-1}}\delta \in \iota^n_j(N^n_j(k_mE_n))$, then we will be done.

Since $\ell_{H_j}(\delta) = \varepsilon p^{n-j-1}+1$ we have 
\begin{eqnarray*}
(\sigma^{p^{j+1}}-1)^{\varepsilon p^{n-1-j-1}+1}\delta =
(\sigma^{p^j}-1)^{\varepsilon p^{n-j-1}+p}\delta =
0 &\quad \mbox{and}\\
(\sigma^{p^{j+1}}-1)^{\varepsilon p^{n-1-j-1}}\delta =
(\sigma^{p^j}-1)^{\varepsilon p^{n-j-1}}\delta \neq 0.&
\end{eqnarray*}
Hence we have
$\ell_{H_{j+1}}(\delta) = \varepsilon p^{n-1-j-1}+1$.  Note that if $N^n_{n-1}(\gamma) = 0$ then $N^n_{n-1}(\delta) = 0$, and that if $E/E_j$ is embeddable then so too is $E/E_{j+1}$.  Hence by induction it follows that $$(\sigma^{p^{j+1}}-1)^{\varepsilon p^{n-1-j-1}} \delta =
\iota_{j+1}^n (N^n_{j+1} (\alpha))$$ for some $\alpha \in k_mE_n$, or equivalently \begin{equation}\label{eq:bottom of delta}(\sigma^{p^j}-1)^{\varepsilon p^{n-1-j}}\delta = (\sigma^{p^j}-1)^{p^{n-j}-p}\alpha.\end{equation}

Unfortunately, $\alpha$ does not generate a submodule long enough to provide our desired equality.  Instead of being length $p^{n-j}-1$ we have $\ell_{H_j}(\alpha) = p^{n-j}-p+1$:
\begin{equation*}
\begin{split}
(\sigma^{p^{j}}-1)^{p^{n-j}-p}\alpha &=
(\sigma^{p^{j}}-1)^{\varepsilon p^{n-1-j}}\delta \neq 0 \quad \mbox{and}\\
(\sigma^{p^{j}}-1)^{p^{n-j}-p+1}\alpha &= (\sigma^{p^j}-1)^{\varepsilon p^{n-1-j}+1} \delta = 0.
\end{split}
\end{equation*}

We use induction to show that the $H_{j+1}$-fixed part of the $\F_p[H_{j+1}]$-submodule $\langle(\sigma^{p^j}-1)\alpha\rangle$ is generated by some $\iota^n_{j+1}(N^n_{j+1}(\beta))$, which will ultimately provide the desired result.  With this goal in mind, we compute $\ell_{H_{j+1}}\left((\sigma^{p^j}-1)\alpha\right)$.  First, we have $$(\sigma^{p^{j+1}}-1)^{p^{n-j-1}-2}(\sigma^{p^j}-1)\alpha = (\sigma^{p^j}-1)^{p^{n-j}-2p+1}\alpha \neq 0,$$ where the inequality follows from the fact that $\ell_{H_j}(\alpha) = p^{n-j}-p+1 > p^{n-j}-2p+1$.  We also have
$$(\sigma^{p^{j+1}}-1)^{p^{n-j-1}-1}(\sigma^{p^j}-1)\alpha = (\sigma^{p^j}-1)^{p^{n-j}-p+1}\alpha = 0,$$ again using $\ell_{H_j}(\alpha) = p^{n-j}-p+1$. Hence we have $\ell_{H_{j+1}}\left((\sigma^{p^j}-1)\alpha\right) = p^{n-j-1}-1$.

Provided $p \neq 3$ or $j \neq n-2$ we have $p^{n-j-1}-1 >2 p^{n-1-j-1}$, so by induction we have $$(\sigma^{p^{j+1}}-1)^{p^{n-j-1}-2} (\sigma^{p^j}-1)\alpha = \iota^n_{j+1}(N^n_{j+1}(\beta))=(\sigma^{p^{j+1}}-1)^{p^{n-j-1}-1}\beta$$ for some $\beta \in k_mE_{j+1}$.  Equivalently, this means \begin{equation}\label{eq:bottom.of.module.is.fixed.part.of.super.long.module}(\sigma^{p^j}-1)^{p^{n-j}-2p}(\sigma^{p^j}-1)\alpha = (\sigma^{p^j}-1)^{p^{n-j}-p}\beta.\end{equation} Hence, recalling Equation (\ref{eq:bottom of delta}) for equality $\star$ below, we have the desired result:
\begin{equation*}\begin{split}
\iota^n_j(N^n_j(\beta)) &=(\sigma^{p^{j}}-1)^{p^{n-j}-1} \beta \\&= (\sigma^{p^j}-1)^{p-1}(\sigma^{p^{j}}-1)^{p^{n-j}-p}\beta\\&=
(\sigma^{p^j}-1)^{p-1}(\sigma^{p^{j}}-1)^{p^{n-j}-2p}(\sigma^{p^j}-1)\alpha\\&=
(\sigma^{p^j}-1)^{p^{n-j}-p} \alpha \\&\stackrel{\star}{=}
(\sigma^{p^{j}}-1)^{\varepsilon p^{n-1-j}}\delta.
\end{split}\end{equation*}

Finally, suppose that $p=3$ and $j=n-2$.  In this case $\ell_{H_{n-2}}(\alpha) = 7$, so that $(\sigma^{3^{n-2}}-1)^6 \alpha \in (k_mE_n)^{H_{n-2}}$.  We also know that $(\sigma^{3^{n-2}}-1)^6\alpha = (\sigma^{3^{n-1}}-1)^2 \alpha = \iota^n_{n-1}(N^n_{n-1}(\alpha))$, so that $(\sigma^{3^{n-2}}-1)^6\alpha \in \ker(N^n_{n-1})$.  Hence Lemma \ref{le:warmup to fixed} gives  $(\sigma^{3^{n-2}}-1)^6\alpha = \iota^n_{n-1}\left(N^n_{n-1}\left(\alpha\right)\right) = \iota^n_{n-2}(h)$ for some $h \in k_mE_{n-2}$, and so there exists $g \in \Gamma(m,n)$ so that $$N^n_{n-1}(\alpha) = \{a_{n-1}\}\cdot g + \iota^{n-1}_{n-2}(h).$$  Now \cite[Prop.~7]{MSS1} provides an element $\chi \in k_1E_n$ with $\ell_{H_{n-1}}(\chi) \leq 2$ and so that $N^n_{n-1}(\chi) = a_{n-1}$.  Note that $g \in k_{m-1}E_{n-1}$ gives
\begin{equation*}
\begin{split}
(\sigma^{3^{n-2}}-1)^{6}\left(\{\chi\}\cdot\iota^n_{n-1}(g)\right) &=(\sigma^{3^{n-1}}-1)^2\left(\{\chi\}\cdot \iota^n_{n-1}(g)\right) \\&= \left((\sigma^{3^{n-1}}-1)^2\{\chi\}\right)\cdot \iota^n_{n-1}(g) = 0.
\end{split}
\end{equation*}
Set $\alpha' = (\sigma^{3^{n-2}}-1)\left(\alpha -\{\chi\}\cdot \iota^n_{n-1}(g)\right)$.  Since $\ell_{H_{n-2}}(\alpha) = 7$ and $\ell_{H_{n-2}}(\{\chi\}\cdot\iota^n_{n-1}(g)) = 6$, this leaves $(\sigma^{3^{n-2}}-1)^5\alpha' = (\sigma^{3^{n-2}}-1)^6 \alpha$; from this it follows that  $\ell_{H_{n-1}}(\alpha') = 2$.  We also have $$N^n_{n-1}\left(\alpha'\right) = (\sigma^{3^{n-2}}-1)\left(\{a_{n-1}\}\cdot g + \iota^{n-1}_{n-2}(h) - \{a_{n-1}\}\cdot g\right)= 0.$$  Hence by induction there exists some element $\beta$ with $$(\sigma^{3^{n-1}}-1)\alpha' = \iota^n_{n-1}\circ N^n_{n-1}\beta = (\sigma^{3^{n-1}}-1)^2\beta = (\sigma^{3^{n-2}}-1)^{6}\beta,$$  which gives
\begin{equation*}
\begin{split}
\iota_{n-2}^n(N^n_{n-2}(\beta)) 
&= (\sigma^{3^{n-2}}-1)^8\beta \\
&= (\sigma^{3^{n-2}}-1)^2(\sigma^{3^{n-2}}-1)^6\beta\\
&= (\sigma^{3^{n-2}}-1)^2 (\sigma^{3^{n-1}}-1)\alpha'\\
&= (\sigma^{3^{n-2}}-1)^6 \alpha\\
&= (\sigma^{3^{n-2}}-1)^{3\varepsilon}\delta.
\end{split}
\end{equation*}
\end{proof}

\begin{corollary}\label{cor:fixedeltsarenorms}
For $\gamma \in k_mE_n$, let $i$ be minimal such that $\gamma \in \iota^n_i(k_mE_i)$.  If $N^n_{n-1}(\gamma) = 0$ and $\ell_G(\gamma)>p^{i-1}$, then $(\sigma-1)^{\ell_G(\gamma)-1}\gamma \in \iota^n_0(N^i_0(k_mE_i))$.
\end{corollary}

\noindent \emph{Note:} When $i< n$, the condition $N^n_{n-1}(\gamma) = 0 $ is trivial.

\begin{proof}In the case $i=n$, the result follows by taking $j=0$ in the previous proposition.  For $i<n$, choose $\hat \gamma \in k_m E_i$ with $\iota^n_i (\hat \gamma) = \gamma$; by Lemma \ref{le:warmup to fixed} we can insist $\ell_{G_i}(\hat \gamma) = \ell_G(\gamma)$.  Then $\ell_{G_i}(\hat \gamma) > p^{i-1}$, and since $E_i/E_0$ is embeddable the previous proposition applied to the extension $E_i/E_0$ gives $$(\sigma-1)^{\ell_{G_i}(\hat \gamma)-1} \hat \gamma\in \iota_0^i(N^i_0(k_mE_i )).$$ Therefore \begin{equation*}\begin{split}(\sigma-1)^{\ell_G(\gamma)-1}\gamma
= \iota_i^n\left((\sigma-1)^{\ell_{G_i}(\hat \gamma)-1}\hat \gamma\right) & \subseteq  \iota^n_i
\left(\iota^i_0\left(N^i_0\left(k_mE_i\right)\right)\right) \\ &= \iota_0^n(N^i_0(k_mE_i))\end{split}\end{equation*} as desired.
\end{proof}

We are now ready to give the ``sparse'' $\F_p[G]$-decomposition of $k_mE_n$ provided by Theorem \ref{th:coarse.decomposition}.

\begin{proof}[Proof of Theorem \ref{th:coarse.decomposition}]
Using the notation and results from the proof of Corollary \ref{cor:choose.your.decomposition}, we only need to verify that $V_{i+1} = V_{p^n}$ for every $i$ satisfying $2p^{n-1}+1 \leq i \leq p^n-1$. This means that we must show that for any $x \in \im(\sigma-1)^{i-1}\cap (k_mE_n)^G$, we also have $x \in \im(\sigma-1)^{p^n-1}$.

Choose an $\alpha_x$ with $(\sigma-1)^{i-1}\alpha_x =x$.  Then $\ell_G(\alpha_x) = i$, and since $i > 2p^{n-1}$ we may apply Proposition \ref{prop:fixedeltsarenorms} (with $j=0$) to conclude that $$x=(\sigma-1)^{i-1}\alpha_x = \iota^n_0(N^n_0(\alpha)) = (\sigma-1)^{p^{n}-1}\alpha$$ for some $\alpha \in k_mE$.  Hence $x \in V_{p^n}$ as desired.
\end{proof}

\section{Proof of Theorem \ref{th:embeddable.decomposition}}\label{sec:proof.of.theorem}

We are now prepared to prove the main result of this paper.  Though the machinery developed thus far applies to all extensions $E/F$ with $\Gal(E/F) \simeq \Z/p^n\Z$ and $\xi_p \in E$ --- assuming that $p>2$ is prime --- the main theorem relies critically on the existence of an exceptional element $a_n$ of $E/F$ which satisfies $a_n^{\sigma-1} \in E^{\times p}$.  More specifically, we use this condition to construct modules $X_i$ which appear in the theorem; this is the only place where the embeddable condition is used. 




Let $a_n$ be an arbitrary exceptional element of $E/F$; Proposition \ref{prop:exceptional.element.criteria} gives a $t$ so that $N^n_{n-1}(a_n^t) \in a_{n-1}E_{n-1}^{\times p}$.  We define the module $X$ as $\{a_n^t\}\cdot \iota^n_{n-1}(\Gamma(m,n))$, and claim that our embeddable condition implies $X \simeq \Gamma(m,n)$ as $\F_p[G]$-modules (the $\F_p[G]$-action on $\Gamma(m,n)$ is induced from its $\F_p[G_{n-1}]$-action).  Since Proposition \ref{prop:Gamma module} shows $\Gamma(m,n) = \oplus_{i=0}^{n-1} \mathcal{Z}_i$, where $\mathcal{Z}_i \subseteq  \iota^n_i(k_mE_i)$ is a direct sum of cyclic submodules of dimension $p^i$, the $\F_p[G]$-isomorphism $X \simeq \Gamma(m,n)$ will be enough to show that the $X_i$ satisfy the necessary conditions.

To show $X \simeq \Gamma(m,n)$, first notice that the Projection Formula (\ref{eq:projection.formula}) shows that $N^n_{n-1}(X) = \{a_{n-1}\}\cdot\Gamma(m,n)$.  To see that $\ker(N^n_{n-1}) \cap X = \{0\}$, notice that for nonzero $g \in \Gamma(m,n)$ we have $N^n_{n-1}\left(\{a_n^t\}\cdot \iota^n_{n-1}(g)\right)= \{a_{n-1}\}\cdot g \neq 0$ by Proposition \ref{prop:Gamma module}(\ref{it:iv}).  Finally, the action of $\sigma$ commutes with $N^n_{n-1}$ and is trivial on $a_{n-1}$ (by (\ref{eq:ai.are.fixed})) as well as $a_n$ (by our embeddability condition together with Proposition \ref{prop:exceptional.element.criteria}).  Hence $N^n_{n-1}$ gives an $\F_p[G]$-isomorphism between $X$ and $\{a_{n-1}\}\cdot \Gamma(m,n)$.  Proposition \ref{prop:Gamma module}(\ref{it:iv}) has already established that $\{a_{n-1}\}\cdot \Gamma(m,n) \simeq \Gamma(m,n)$ as $\F_p[G_{n-1}]$-modules, completing the proof of the claim.

Now let $\mathcal{I}_n$ be an $\F_p$-basis for $\iota_{0}^n(N_0^n(k_mE_n))$, and for each $0 \leq i < n$ let $\mathcal{I}_i$ be an $\F_p$-basis for a complement of $\iota_0^n(N^{i+1}_0(k_mE_{i+1}))$ within $\iota_0^n(N^i_0(k_mE_i))$.  For each $x \in \mathcal{I}_i$, $1 \leq i \leq n$, choose an element $\alpha_x \in k_mE_i$ so that $x = \iota_0^n(N^i_0(\alpha_x))$, and define $Y_i = \sum_{x \in \mathcal{I}_i} \langle \alpha_x \rangle$.

As in the proof of Corollary \ref{cor:choose.your.decomposition}, the generator $\alpha_x$ corresponding to $x \in \mathcal{I}_i$ has $$\langle \alpha_x \rangle^G = \langle \iota^n_0(N^i_0(\alpha_x)) \rangle = \langle (\sigma-1)^{p^i-1}\alpha_x \rangle = \langle x \rangle.$$ By construction, the elements of $\cup_i \mathcal{I}_i$ are $\F_p$-independent, and so the Exclusion Lemma (\ref{le:exclusion}) shows $$\sum_{i=0}^{n}\sum_{x \in \mathcal{I}_i} \langle \alpha_x \rangle =  \bigoplus_{i=0}^{n}\bigoplus_{x \in \mathcal{I}_i} \langle \alpha_x \rangle.$$  Since $\iota_0^n \circ N^i_0$ has the same action on $\iota_i^n(k_mE_i)$ as $(\sigma-1)^{p^i-1}$, the modules $Y_i$ satisfy the appropriate conditions.

We have left to show that the $X_i$ modules are independent from the $Y_i$ modules.  The Exclusion Lemma (\ref{le:exclusion}) says we can check independence by looking at the intersection of the corresponding fixed modules.  Recall, however, that $X^G \cap \ker(N^n_{n-1}) = \{0\}$, whereas $Y_i^G \subseteq  \iota_0^n(N^i_0(k_mE_i)) \subseteq  \ker(N^n_{n-1})$.  Hence we conclude that
$$J = \left(\bigoplus_{i=0}^{n-1} X_i\right) + \left(\bigoplus_{i=0}^n Y_i\right) = \left(\bigoplus_{i=0}^{n-1} X_i\right) \oplus \left(\bigoplus_{i=0}^n Y_i\right).$$


Our goal is to show that $k_mE_n = J.$  To do this, recall the notation $V_\ell = \im\left((\sigma-1)^{\ell-1}\right) \cap (k_mE_n)^G$. We shall prove that for each $0 \leq i \leq n$ and $1 \leq j \leq p^{i+1}-p^i$,
\begin{equation}\label{eq:collapsing.of.filtration}
V_{p^i+j} \subseteq  \im\left((\sigma-1)^{p^{i+1}-1}\right) \cap J^G.
\end{equation}  Inasmuch as the right side of this expression is visibly in $V_{p^{i+1}}$, and since we have $V_{p^{i+1}} \subseteq  V_{p^i+j}$ automatically, this condition will ensure that $V_{p^i+j} = V_{p^{i+1}}$.  According to Corollary \ref{cor:general.decomposition}, this implies that the module structure of $k_mE_n$ will contain only cyclic summands of dimension $p^k$, $0 \leq k \leq n$.  Condition (\ref{eq:collapsing.of.filtration}) will also show that $$V_{p^{i}} = \im\left((\sigma-1)^{p^i-1}\right) \cap J^G = \bigoplus_{k \geq i} X_k^G \oplus \bigoplus_{k \geq i} Y_k^G,$$ from which our construction of the summands $X_i$ and $Y_i$, together with Corollary \ref{cor:choose.your.decomposition}, will show that $k_mE_n \simeq J$.

To verify this condition, suppose that $f = (\sigma-1)^{p^i+j-1}\gamma \in (k_mE_n)^G$.  Now if $p^i+j > p^n-p^{n-1}$, then this implies $\ell_G(\gamma) > 2p^{n-1}$.  Hence taking $j=0$ in Proposition \ref{prop:fixedeltsarenorms} shows that $$f \in \iota_0^n(N_0^n(k_mE_n)) = \im\left((\sigma-1)^{p^n-1}\right) \cap (k_mE_n)^G.$$  In this case recall that $Y_n^G = \langle \mathcal{I}_n \rangle = \iota^n_0(N^n_0(k_mE_n))$ by construction, and so $f \in Y_n^G \subseteq  \im\left((\sigma-1)^{p^{n}-1}\right) \cap J^G$ as desired. Otherwise we have $p^i+j \leq p^n-p^{n-1}$, meaning that $(\sigma-1)^{p^n-p^{n-1}}\gamma = \iota_{n-1}^n(N^n_{n-1}(\gamma)) = 0$.  Hence from Exact Sequence (\ref{eq:kummer.for.k.theory}) we must be in the case that $N^n_{n-1}(\gamma)  = \{a_{n-1}\}\cdot \iota_{n-1}^n(g)$, where $g \in \Gamma(m,n)$.

By construction of the module $X$, there exists a unique $x \in X$ --- possibly zero --- so that $N^n_{n-1}(x) = N^n_{n-1}(\gamma)$. Moreover, since $X \simeq \Gamma(m,n)$ we must have $\ell_G(x) = \ell_{G_{n-1}}(g)$.  Notice that since \begin{equation*}\begin{split}0&=N^n_{n-1}\left((\sigma-1)^{\ell(\gamma)} \gamma\right) 
= (\sigma-1)^{\ell(\gamma)} \left(\{a_{n-1}\}\cdot \iota^n_{n-1}(g)\right)\end{split}\end{equation*} and $\{a_{n-1}\}\cdot \Gamma(m,n) \simeq \Gamma(m,n)$ by Proposition \ref{prop:Gamma module}, we must then have $\ell_G(x) = \ell_{G_{n-1}}(g) \leq \ell_{G}(\gamma)$.  Hence the element $\gamma - x$ has trivial image under the map $N^n_{n-1}$, and moreover $\ell_G(\gamma - x) \leq \max\{\ell_G(\gamma),\ell_G(x)\} = \ell_G(\gamma)$.

Suppose first that $\ell_G(\gamma-x) < \ell_G(\gamma)$.  In this case it follows that $\ell_G(x) = \ell_G(\gamma)$, and indeed that $f = (\sigma-1)^{p^i+j-1}x$.  Hence we have $f \in \im\left((\sigma-1)^{p^i+j-1}\right) \cap X^G$.  But notice that since $X$ is a direct sum of cyclic submodules of dimension $p^k$, where $0 \leq k \leq n-1$, this in turn implies that $$f \in \im\left((\sigma-1)^{p^{i+1}-1}\right) \cap X^G \subseteq  \im\left((\sigma-1)^{p^{i+1}-1}\right) \cap J^G.$$

Finally, we are left with the case that $\ell_G(\gamma-x) = \ell_G(\gamma)$.  In this case we have $\gamma - x \in \ker(N^n_{n-1}) \cap \left((k_mE_n)^{H_{i+1}} \setminus (k_mE_n)^{H_i}\right)$.  Hence Lemma \ref{le:warmup to fixed} and the fact that $\iota_i^n(k_mE_i) \subseteq (k_mE_n)^{H_i}$ implies that $\gamma - x \in \im(\iota^n_{i+1})\setminus \im(\iota^n_{i})$, and Corollary \ref{cor:fixedeltsarenorms} shows that $(\sigma-1)^{p^i+j-1}(\gamma-x) = \iota_0^n(N^{i+1}_0(\alpha))$ for some $\alpha \in k_mE_{i+1}$, so that $$f = (\sigma-1)^{p^i+j-1}\big(x+(\gamma-x)\big) = (\sigma-1)^{p^i+j-1}x + \iota_0^n(N^{i+1}_0(\alpha)).$$

Considering that $\iota^n_0 \circ N^{i+1}_0$ is represented by $(\sigma-1)^{p^{i+1}-1}$ on $\iota_{i+1}^n(k_mE_{i+1})$, it is easy to see that $$\iota_0^n(N^{i+1}_0(\alpha)) \in \im\left((\sigma-1)^{p^{i+1}-1}\right) \cap J^G.$$  On the other hand, since $\ell_G(x) \leq \ell_G(\gamma)$ we see that $(\sigma-1)^{p^i+j-1}x \in \im\left((\sigma-1)^{p^i+j-1}\right) \cap X^G$; since $X$ is composed of cyclic indecomposables of prime-power dimension, it therefore follows that $$(\sigma-1)^{p^i+j-1}x \in \im\left((\sigma-1)^{p^{i+1}-1}\right) J^G.$$  Combining these two observations, we have $f\in \im\left((\sigma-1)^{p^{i+1}-1}\right) \cap J^G$ as desired.  


\begin{thebibliography}{LLMS3}
\bibitem[A]{A} A.A.~Albert.  On cyclic fields. \emph{Trans. Amer. Math. Soc.} {\bf 37} (1935), 454--462.

\bibitem[BLMS]{BLMS} D.~Benson, N.~Lemire, J.~Min\'a\v{c}, and
J.~Swallow. Detecting pro-$p$-groups that are not absolute Galois
groups. \emph{J.~Reine Angew.~Math.} {\bf 613} (2007), 175--191.

\bibitem[FV]{FV} I.~Fesenko and S.~Vostokov. \textit{Local
fields and their extensions}, 2nd ed. Translations of Mathematical
Monographs 121. Providence, RI: American Mathematical Society, 2002.

\bibitem[HW]{HW} Ch.~Haesemeyer and C.~W.~Weibel. Norm varieties
and the chain lemma (after Markus Rost). \emph{Proc.~Abel Symp.}, to appear.  

\bibitem[LLMS1]{LLMS2} J. Labute, N.~Lemire, J.~Min\'a\v{c},
and J.~Swallow. Demu\v{s}kin groups, Galois modules, and the elementary type conjecture. \emph{J.~Algebra} {\bf 304} (2006), 1130--1146.

\bibitem[LLMS2]{LLMS3} J. Labute, N.~Lemire, J.~Min\'a\v{c},
and J.~Swallow. Cohomological dimension and Schreier's
formula in Galois cohomology. \emph{Canad. Math. Bull.} {\bf 50}
(2007), no. 4, 588--593.

\bibitem[LMS]{LMS} N.~Lemire, J.~Min\'a\v{c}, and J.~Swallow.
When is Galois cohomology free or trivial? \emph{New York J.~Math.} {\bf 11} (2005), 291--302.

\bibitem[LMS2]{LMS3} N.~Lemire, J.~Min\'a\v{c}, and J.~Swallow.
Galois module structure of Galois cohomology and partial
Euler-Poincar\'e characteristics. \emph{J.~Reine Angew.~Math.}
{\bf 613} (2007), 147--173.

\bibitem[LMSS]{LMSS} N.~Lemire, J.~Min\'a\v{c}, A.~Schultz, and J.~Swallow.  Hilbert 90 for Galois cohomology. \emph{Comm. in Alg.}, to appear.

\bibitem[MS1]{MS1} A.~Merkurjev and A.~Suslin. $K$-cohomology and Severi-Brauer varieties and the norm residue homomorphism.  \emph{Math. USSR Izvestiya} \textbf{21} (1983), 307--340.

\bibitem[MS2]{MS2} J.~Min\'a\v{c} and J.~Swallow. Galois module structure of $p$th-power classes of extensions of degree $p$. \emph{Israel J.~Math}. \textbf{138} (2003), 29--42.

\bibitem[MS3]{MS3} J.~Min\'a\v{c} and J.~Swallow.  Galois embedding
problems with cyclic quotient of order $p$. \textit{Israel J.~Math}.
\textbf{145} (2005), 93--112.

\bibitem[MSS1]{MSS1} J.~Min\'a\v{c}, A.~Schultz, and J.~Swallow. Galois module structure of the $p$th-power classes of cyclic
extensions of degree $p^n$.  \emph{Proc.~London Math.~Soc.}
\textbf{92} (2006), no.~2, 307--341.

\bibitem[MSS2]{MSS2} J.~Min\'a\v{c}, A.~Schultz, and J.~Swallow. Automatic realizations of Galois groups with cyclic quotient of order $p^n$. \emph{J.~Th\'{e}or.~Nombres Bordeaux} \textbf{20} (2008), no.~2, 419--430.

\bibitem[NSW]{NSW} J.~Neukirch, A.~Schmidt, and K.~Wingberg.
\textit{Cohomology of number fields}.  Berlin: Springer-Verlag,
2000.

\bibitem[R1]{Ro1} M.~Rost. Chain lemma for symbols.
Available at www.math.uni-bielefeld.de/$\sim$rost/chain-lemma.html.

\bibitem[R2]{Ro2} M.~Rost.  On the basic
correspondence of a splitting variety. Available at
www.math.uni-bielefeld.de/$\sim$rost/chain-lemma.html.

\bibitem[SJ]{Su} A.~A.~Suslin and S.~Joukhovitski. Norm
varieties. {\it J.~Pure Appl.~Algebra} {\bf 206} (2006),
235--276.

\bibitem[V1]{Vo1} V.~Voevodsky. Motivic cohomology with
$\Z/2$-coefficients. \emph{Publ. Inst. Hautes \'Etudes Sci.} {\bf 98}
(2003), 59--104.

\bibitem[V2]{Vo2} V.~Voevodsky. On motivic cohomology with
$\Z/l$-coefficients. $K$-theory preprint archive 639.
Available at www.math.uiuc.edu/K-theory/0639/.

\bibitem[W1]{W1} C.~W.~Weibel. The norm residue isomorphism
theorem. Available at www.math.rutgers.edu/$\sim$weibel/papers.html.

\bibitem[W2]{W2} C.~W.~Weibel. Axioms for the norm residue isomorphism.  \emph{K-theory and noncommutative Geometry}, G.~Corti\~{n}as, J.~Cuntz, M.~Karoubi, R.~Nest, and C.~Weibel, eds., pp.~427--435. EMS Series of Congress Reports.  Z\"urich: European Math. Soc. Pub. House, 2008.

\bibitem[W3]{W3} C.~W.~Weibel. The proof of the Bloch-Kato
Conjecture. \emph{ICTP Lecture Notes Series} {\bf 23} (2008), 1--28. 

\end{thebibliography}
\end{document}